\documentclass[conference]{ieeeconf}
\IEEEoverridecommandlockouts
\usepackage{cite}
\usepackage{amsmath,amssymb,amsfonts,amsthm}
\usepackage{graphicx}
\usepackage{textcomp}
\usepackage{xcolor}
\usepackage{algorithm}
\usepackage[noend]{algpseudocode}
\usepackage{comment}
\usepackage{bm}
\usepackage[T1]{fontenc}
\let\labelindent\relax
\usepackage{enumitem}

\newtheorem{theorem}{Theorem}
\newtheorem{corollary}[theorem]{Corollary}

\newtheorem{assumption}{Assumption}

\def\BibTeX{{\rm B\kern-.05em{\sc i\kern-.025em b}\kern-.08em
    T\kern-.1667em\lower.7ex\hbox{E}\kern-.125emX}}

\title{\LARGE \bf Cooling Under Convexity: An Inventory Control Perspective on Industrial Refrigeration}

\author{Vade Shah*, Yohan John*, Ethan Freifeld, Lily Y. Chen, and Jason R. Marden
\thanks{This paper was supported by the California Energy Commission, project \#EPIC-24-012.}
\thanks{V. Shah ({\tt\small vade@ucsb.edu}), Y. John, E. Freifeld, L. Y. Chen, and J. R. Marden are with the Center for Control, Dynamical Systems, and Computation at the University of California, Santa Barbara, CA. * denotes equal contribution.}%
}

\begin{document}

\maketitle 

\begin{abstract}
Industrial refrigeration systems have substantial energy needs, but optimizing their operation remains challenging due to the tension between minimizing energy costs and meeting strict cooling requirements. Load shifting—strategic overcooling in anticipation of future demands—offers substantial efficiency gains. This work seeks to rigorously quantify these potential savings through the derivation of optimal load shifting policies. Our first contribution establishes a novel connection between industrial refrigeration and inventory control problems with convex ordering costs, where the convexity arises from the relationship between energy consumption and cooling capacity. Leveraging this formulation, we derive three main theoretical results: (1) an optimal algorithm for deterministic demand scenarios, along with proof that optimal trajectories are non-increasing (a valuable structural insight for practical control); (2) performance bounds that quantify the value of load shifting as a function of cost convexity, demand variability, and temporal patterns; (3) a computationally tractable load shifting heuristic with provable near-optimal performance under uncertainty. Numerical simulations validate our theoretical findings, and a case study using real industrial refrigeration data demonstrates an opportunity for improved load shifting.
\end{abstract}

\section{Introduction}

Industrial refrigeration accounts for nearly 8.4\% of energy expenditure in the U.S.~\cite{EIA_MECS}, presenting significant opportunities for reducing costs and emissions through intelligent control. One particularly valuable strategy is \emph{load shifting}--adjusting the timing of energy consumption in anticipation of high-demand periods. Since refrigerated product serves as a thermal battery, it can be cooled well before future heat loads arrive and maintain low temperatures thereafter.

\begin{figure}
    \centering
    \includegraphics[width=\linewidth]{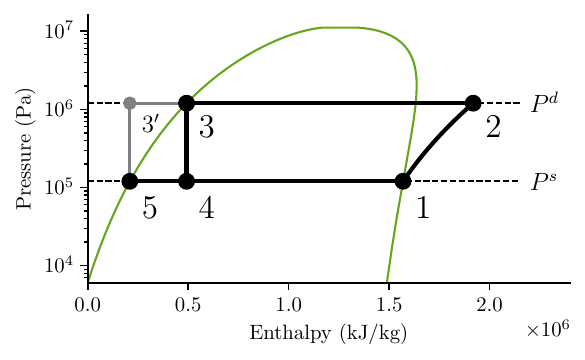}
    \caption{The pressure-enthalpy ($P$-$h$) diagram for ammonia, the most common industrial refrigerant. The saturation dome (green) separates phases of matter. In the vapor compression refrigeration cycle (solid black), heat is removed by cycling a refrigerant through the five points on the diagram in numerical order. The suction pressure $P^s$ and discharge pressure $P^d$ uniquely determine all five points. We consider a widely implemented configuration that includes a pressure vessel corresponding to point $4$ which separates the refrigerant into saturated liquid and vapor phases; our model also applies to the standard refrigeration cycle with liquid subcooling in the condenser to point $3'$ (gray). We assume that (i) the system is at thermodynamic steady-state, (ii) the compressor performs isentropic compression, (iii) the expansion valve performs isenthalpic expansion, (iv) the condenser/evaporator return precisely saturated liquid/vapor phases, and (v) pressure drops in the pipes and devices (other than the expansion valve) are neglected.
    }
    \label{fig:P_h_cycle}
\end{figure}

Industrial refrigeration systems operate via the \emph{vapor-compression refrigeration cycle} shown in Figure~\ref{fig:P_h_cycle}. The cycle consists of four stages: (i) compression of low-pressure vapor to high-pressure vapor, (ii) condensation to high-pressure liquid, (iii) expansion to low-pressure liquid-vapor mixture, and (iv) evaporation back to low-pressure vapor. Two stages are critical for the control problem we address. In the \emph{evaporation} stage (5$\rightarrow$1 in Figure \ref{fig:P_h_cycle}), the refrigerant removes heat from the refrigerated space. In the \emph{compression} stage (1$\rightarrow$2), which consumes the most energy~\cite{Konda2024}, work is done on the refrigerant, enabling it to reject heat to the atmosphere. 

A critical control parameter in industrial refrigeration is \emph{suction pressure}~\cite{Jain2011, Stoecker1998}, which determines not only the rate of heat removal in the evaporation stage, but also the work done on the refrigerant in the compression stage. The relationship between suction pressure and system performance creates a fundamental tradeoff: lower suction pressure increases the cooling rate but requires more work, while higher suction pressure requires less work but may provide insufficient cooling to meet demand. Accordingly, much prior work has focused on dynamic suction pressure control with the goal of minimizing energy consumption while satisfying cooling requirements~\cite{Hovgaard2013, Larsen2006, Koeln2014, John2025}. 
These works have demonstrated effective suction pressure control strategies, but they focus primarily on practical implementation rather than rigorous performance guarantees. Thus, our work seeks to address this gap by developing a formal understanding of various suction pressure control techniques, and in particular, explicitly quantifying the performance gap between optimal strategies that perform load shifting and shortsighted, suboptimal ones that do not.



\subsection{Contributions}

Our first result provides a novel lens for industrial refrigeration control by connecting it to inventory control, a foundational problem in operations research. The thermal mass of refrigerated product corresponds to inventory, the heat removal rate serves as the control variable, and the work-heat relationship defines the ordering costs. The unique features are that the relationship between heat removal and compressor work exhibits convexity, holding costs are negligible, and backlogs cannot be tolerated. By incorporating these features into the framework of inventory control, we arrive at three novel theoretical contributions applicable to suction pressure control:

\begin{enumerate}[leftmargin=*]
\item First, we consider the deterministic setting where future demand is perfectly known. We provide an efficient algorithm and prove in Theorem~\ref{thm:deterministic} that it yields the optimal control policy. Importantly, we also establish a structural result: the optimal heat removal is non-increasing over time, which is a valuable insight for implementing load shifting in practice.  

\item Second, we analyze the stochastic setting with random demands. Theorem~\ref{thm:lb_myopic} establishes a lower bound on the cost savings achievable through load shifting compared to a myopic policy that only minimizes immediate costs. This bound is parameterized by the convexity of the cost function and the variability of incoming demands, revealing when load shifting offers the greatest value.

\item Third, recognizing that computing the optimal policy under uncertainty may be computationally prohibitive, we present a tractable load shifting heuristic inspired by our deterministic algorithm. Theorem~\ref{thm:ub_heuristic} provides an upper bound on the performance gap between this heuristic and the optimal policy, demonstrating its effectiveness as a practical alternative to dynamic programming.
\end{enumerate}


Finally, we validate our theoretical findings through numerical simulations and present a case study using real industrial refrigeration data that indicates opportunities for improvement in existing control implementations. 

\section{Refrigeration As Inventory Control}


In this section, we establish an important connection between suction pressure control in industrial refrigeration and classical inventory control. In dynamic suction pressure control, the goal is to minimize the total work required of the compressor over time while removing enough heat at the evaporator to maintain temperature requirements. Suction pressure optimization thus depends heavily on the relationship between heat removal and work. We study this relationship using the $P$-$h$ diagram.

Let $h_i$ denote the enthalpy at point $i \in \{1, 2, 3, 4, 5\}$ in the vapor-compression cycle shown in Figure~\ref{fig:P_h_cycle}. For a fixed\footnote{Discharge pressure is known to have a smaller impact on efficiency and is typically controlled in relation to ambient conditions \cite{Stoecker1998}; hence, we focus on suction pressure as our primary control input and consider any reasonable fixed discharge pressure throughout. Note that $G$ is convex for a wide range of discharge pressures (Figure~\ref{fig:power_heat_curves}).} discharge pressure $P^d$, the compressor work $\overline{W}(P^s) \triangleq \dot{m}_c (h_2 - h_1)$ is a function of suction pressure $P^s$ and refrigerant mass flow rate $\dot{m}_c$ through the compressor.
The evaporator heat absorption $\overline{H}(P^s) \triangleq \dot{m}_e (h_1 - h_5)$ is a function of suction pressure $P^s$ and mass flow rate $\dot{m}_e$ through the evaporator. Under our assumptions, a steady-state energy balance yields a relationship between $\dot{m}_c,\dot{m}_e$; therefore, only one of the two quantities can be chosen independently. It is well-known that compressors operate most efficiently at maximum mass flow rate~\cite{Stoecker1998}, so it is desirable to first utilize the thermodynamic degrees of freedom before resorting to varying mass flow rate. Hence, in what follows, we remain agnostic to system scale, i.e., we assume that the compressor operates at its maximum mass flow rate (which, in turn, determines the evaporator mass flow rate), and focus exclusively on thermodynamic effects\footnote{Note that the convexity highlighted in Figure \ref{fig:power_heat_curves} is preserved with a fixed mass flow rate, so this does not affect any of our conclusions.}. Specifically, we consider the normalized quantities \emph{specific work} 
\begin{equation*}
W(P^s) \triangleq h_2 - h_1
\end{equation*}
and \emph{specific heat absorption} 
\begin{equation*}
    H(P^s) \triangleq h_1 - h_5.
\end{equation*}

We are now equipped to define the optimization problem of interest in this paper. Suppose that in each period $k \in \{0, \dots, N - 1\}$, $w_k$ units of heat load arrive at the refrigerated space. The evaporator must meet that heat load either in the timestep it arrives or by precooling the space ahead of time, or load shifting, by lowering the suction pressure further than needed. The goal of the \emph{dynamic suction pressure optimization} problem is to ensure that the evaporator can satisfy the heat load requirements while minimizing the cumulative work performed by the compressor:
\begin{equation}\label{eq:SPO_nominal}
\begin{aligned}
    \min_{\bm{P}^s} \quad & \sum_{k=0}^{N-1} W(P^s_k) \\
    \textrm{s.t.} \quad & \sum_{j = 0}^k H(P^s_k) \geq \sum_{j = 0}^k w_j, \quad k \in \{0, \dots, N - 1\}.
\end{aligned}
\end{equation}
Here, the constraints imply that at each period, we must set suction pressure so as to remove \emph{at least} the cumulative heat load that has arrived so far. 

We will now rewrite this optimization problem in a more convenient form. Consider the following change of variables:
\begin{align*}
    u_k &\triangleq H(P_k^s) ,\\
    x_0 &= 0, \\
    x_{k + 1} &= x_k + u_k - w_k, \quad k \in \{0, \dots, N - 1 \}.
\end{align*}
Here, $u_k$ represents the heat removed at time $k$, and $x_k$ represents the `precooling buffer' available in the system at time $k$. Using these variables, we can rewrite \eqref{eq:SPO_nominal} as
\begin{equation}\label{eq:SPO}
\begin{aligned}
    \min_{\bm{u}} \quad & \sum_{k=0}^{N-1} G(u_k) \\
    \textrm{s.t.} \quad & x_{k + 1} = x_k + u_k - w_k, \quad k \in \{0, \dots, N - 1\}, \\
    & x_0 = 0, \\
    & x_k \geq 0, \quad k \in \{1, \dots, N\}.
\end{aligned}
\end{equation}
where
\begin{equation}
    G(u) \triangleq W  \circ H^{-1} (u).
\end{equation}
directly maps heat removal to compressor work. The functions $W,H,G$ can be calculated using the thermodynamic equations of state implemented in~\cite{CoolProp}; here, note that $H$ is strictly decreasing in $P^s$, so it is invertible.

The formulation~\eqref{eq:SPO} is convenient because it resembles an \emph{inventory control} problem. In inventory control, one seeks to manage the inventory level of a product so as to meet demand while minimizing the costs of ordering additional product, holding excess inventory, and experiencing backlogs by having insufficient inventory. 
The analogy is that the precooling buffer $x_k$, heat absorption $u_k$, and heat load $w_k$ can be thought of as inventory, orders, and demands, respectively. However, \eqref{eq:SPO} differs from typical inventory control in two major ways. First, whereas inventory control problems usually have costs for holding excess product and for carrying insufficient product,
our model imposes no penalty for maintaining a precooling buffer, i.e., zero holding costs, and instead \emph{requires} that the precooling buffer is always nonnegative, i.e., infinite backlog costs. 
Second, in our refrigeration model, the cost of ordering is $G(u)$. Figure~\ref{fig:power_heat_curves} shows a plot of $G$ for ammonia, the most common industrial refrigerant\footnote{Similar relationships hold for other common refrigerants including R134a and carbon dioxide.}. It turns out that for a wide range of discharge pressures, this ordering cost function is convex.

The dynamic suction pressure optimization problem can thus be viewed as an inventory control problem with zero holding costs, infinite backlog costs, and convex ordering costs. Henceforth, we study this problem using the framework and terminology of inventory control, and draw connections back to refrigeration when appropriate.

\begin{figure}
    \includegraphics[width=\linewidth]{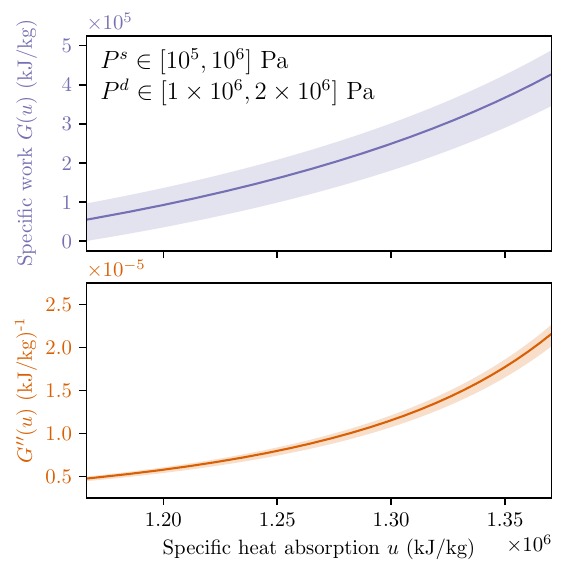}
    \caption{Top: The function $G(u)$, which relates specific heat absorption to specific work for ammonia. Bottom: The second derivative $G''(u)$; observe that it is strictly positive. The solid line in each plot corresponds to $P^d = 1.5 \times 10^6$ Pa; the shaded surrounding regions illustrate that a similar convex relationship holds for a wide range of practical discharge pressures.}
    \label{fig:power_heat_curves}
\end{figure}

\section{Inventory System Model}


In this section we present our model of an inventory system motivated by suction pressure control in industrial refrigeration. First, we assume that the ordering cost function satisfies the following:
\begin{assumption}\label{assumption:convex_cost}
    $G: \mathbb{R}_{\geq 0} \to \mathbb{R}_{\geq 0}$ is a strictly convex, increasing function with $G(0) = 0$. 
\end{assumption}

\noindent Additionally, we assume that the demand satisfies the following:
\begin{assumption}\label{assumption:positive_demand}
    The demand $w_k$ is sampled from a distribution $W_k$ with mean $\mu_k$, variance $\sigma_k^2$, and support $[\mu_k - \Delta, \mu_k + \Delta]$, where $\Delta \geq 0$ and $\mu_k - \Delta > 0$ for all $k \in \{0, \dots, N - 1\}$.
\end{assumption}
\noindent Hence, akin to heat loads in refrigeration, the demand in each period causes the inventory level in the system to either remain constant or decrease. 

An inventory system is completely described by the tuple $P = (G, \bm{\mu}, \bm{\Sigma}, \Delta)$, where $\bm{\mu}$ is an $N$-dimensional vector of means and $\bm{\Sigma}$ is a $N \times N$ covariance matrix. We assume that in each period, the order amount is chosen according to a policy $\bm{\pi} = \{ \pi_0, \dots, \pi_{N-1} \}$,
a sequence of functions $\pi_k : \mathbb{R}_{\geq 0} \to \mathbb{R}_{\geq 0}$ that map the current state $x_k$ to an order amount $u_k$. The expected cost of a policy $\bm{\pi}$ from an initial condition $x_0$ for a given problem $P$ is
\begin{equation}\label{eq:policy_cost}
    J_{\bm{\pi}}(x_0 \, | \, P) \triangleq \sum_{k = 0}^{N - 1} \mathbb{E} \left[G( \pi_k(x_k)) \right].
\end{equation}

The optimal policy $\bm{\pi}^*$ is the one that minimizes the expected costs \eqref{eq:policy_cost} subject to the constraint $x_k \geq 0$ for all $k$.
Importantly, an optimal policy in this setting may perform \emph{load shifting}.
In contrast, a naive policy that performs no load shifting, commonly referred to as the \emph{myopic} policy $\bm{\pi}^{\textup{my}}$, is defined as follows:
\begin{equation}
    \bm{\pi}_k^{\textup{my}} \triangleq \max\{0, \mu_k + \Delta - x_k\}.
\end{equation}
The myopic policy orders the minimum amount necessary to ensure a nonnegative inventory level; in general, this strategy need not be optimal. The focus of this paper is on understanding the performance gap between the optimal and myopic policies in inventory systems with convex ordering cost functions to understand the value of load shifting in industrial refrigeration. We begin by studying this question in the deterministic setting.

\section{Deterministic Inventory Systems}\label{sect:deterministic}

When the demands that enter the system are known a priori (i.e., $\Delta = 0$), the control problem is simply~\eqref{eq:SPO}, an open-loop convex optimization problem. Below, we review the relevant literature before presenting our main result.

\subsection{Literature Review}

The problem \eqref{eq:SPO} is a convex optimization problem with an additively separable objective function and so-called linear ascending constraints. Though our motivation stems from suction pressure optimization in industrial refrigeration, similar problems arise in signal processing~\cite{DAmico2014} and communication systems~\cite{Viswanath2002}.
Early work on separable convex optimization demonstrates that problems with linear constraint matrices whose subdeterminants are small (a condition met by linear ascending constraints) can be solved in polynomial time, and in the specific setting of polynomial objective functions with linear ascending constraints, can be solved by a greedy algorithm~\cite{Morton1985, Hochbaum1990}. Closely related to our work are more recent studies~\cite{Padakandla2010, Vidal2019} which provide algorithms for solving these kinds of problems and characterize important monotonicity and convexity properties of the optimal value. Our work diverges from the existing literature in that we consider the specialized setting where the separable objective function is the sum of \emph{identical} convex functions. Although this setting is more restrictive, it allows us to identify important structural and monotonicity properties of the optimal solution.

\subsection{Optimal Load Shifting}

When demands are known perfectly a priori, our first result establishes that Algorithm~\ref{alg:load_shift} is optimal for solving~\eqref{eq:SPO}:
\begin{theorem}\label{thm:deterministic}
    For any inventory system $P = (G, \bm{\mu}, \bm{\Sigma}, \Delta)$ that satisfies Assumptions \ref{assumption:convex_cost} and \ref{assumption:positive_demand} with $\Delta = 0$, $\bm{u}^* = \text{\textsc{LoadShift}}(\bm{\mu}, 0, 0)$ is the optimal solution to \eqref{eq:SPO}. Furthermore, the optimal trajectory is non-increasing, i.e.,
    \begin{equation}
        u^*_{k+1} \leq u^*_k, \quad \forall k \in \{0,\dots,N-2\}.
    \end{equation}
\end{theorem}

\begin{algorithm}
\caption{Load shifting algorithm.}\label{alg:load_shift}
\begin{algorithmic}
\Procedure{LoadShift($\bm{w}$, $\Delta$, $x$)}{}
\State $k \gets 0$, $w_0 \gets w_0 + \Delta$
\While{$k \leq N - 1$}
\State $\mathcal{S} = \arg \max_{j \in \{1, \dots, N - k + 1\}}((\sum_{i=k}^{k+j-1} w_i) - x)/j$
\State $j \gets k + \max \mathcal{S}$
\For{$i \in \{k, \dots, j - 1\}$}
\State $u_i \gets \frac{1}{j - k} \sum_{i=k}^{j-1} w_i$
\EndFor
\State $k \gets j$
\EndWhile
\State \Return $\bm{u}$
\EndProcedure
\end{algorithmic}
\end{algorithm}

\begin{proof}
    First, we remark that Algorithm~\ref{alg:load_shift} terminates in at most $N$ iterations because in each iteration, $k$ increases by at least 1. Next, we establish that the algorithm returns a feasible solution. Suppose that the algorithm iterates through the while loop $n \leq N$ times. Let $k_0, \dots, k_{n}$ denote the values that $k$ takes over the course of the algorithm (here, $k_0 = 0$ and $k_{n} = N$). Observe that
    \begin{equation}\label{eq:alg_tight}
        \sum_{j = k_0}^{k_1 - 1} u_j = \sum_{j = k_0}^{k_1 - 1} \frac{ \sum_{i=k_0}^{k_1-1} w_i}{k_1 - k_0 } = \sum_{j=k_0}^{k_1-1} w_j,
    \end{equation}
    meaning that the constraint corresponding to $k_1 - 1$ is met with equality. This implies that for all $k > k_1 - 1$, the constraints can be rewritten as
    \begin{equation}\label{eq:alg_reindex_constraints}
        \sum_{j = k_0}^{k} u_j \geq \sum_{j=k_0}^{k} w_j
        \iff
        \sum_{j = k_1}^{k} u_j \geq \sum_{j=k_1}^{k} w_j.
    \end{equation}
    From \eqref{eq:alg_tight}, it is readily verified that the constraints corresponding to $k_2 - 1, \dots, k_{n} - 1$ are also met with equality. Consider the constraint corresponding to some $k$ for which $k_l \leq k < k_{l + 1} - 1$ for some $l \in \{1, \dots, n \}$. We have that
    \begin{equation*}
        \sum_{j = k_l}^{k} u_j = \sum_{j = k_l}^{k} \frac{ \sum_{i=k_l}^{k_{l + 1}-1} w_i}{k_{l + 1} - k_l } \geq \sum_{j = k_l}^{k} \frac{ \sum_{i=k_l}^{k} w_i}{k - k_l + 1} = \sum_{i=k_l}^{k} w_i,
    \end{equation*}
    where the inequality follows from the fact that $k_{l + 1} \in \arg \max_{j \in \{1, \dots, N - k + 1\}} \frac{1}{j} \sum_{i=k_l}^{k_l+j-1} w_i$ by construction in Algorithm~\ref{alg:load_shift}. Thus, we have that every constraint is satisfied.
    
    Last, we establish optimality. We first remark that an optimal solution must meet the final constraint with equality; from \eqref{eq:alg_tight}, this is true of the solution returned by the algorithm. We claim that an optimal solution must also meet the constraint corresponding to $k_1 - 1$ with equality. To see this, suppose that $\tilde{\bm{u}}$ is an optimal solution for which $\sum_{j = k_0}^{N} \tilde{u}_j = \sum_{j = k_0}^{N} w_j$ and $\sum_{j = k_0}^{k_1 - 1} \tilde{u}_j > \sum_{j = k_0}^{k_1 - 1} w_j$. These two facts imply
    \begin{equation*}
        \sum_{j = k_1}^{N - 1} \tilde{u}_j
        < \sum_{j = k_1}^{N - 1} w_j \quad \text{and} \quad \tilde{u}_{k'} > \frac{1}{k_1 - k_0} \sum_{j=k_0}^{k_1 - 1} w_j
    \end{equation*}
    for some $k' \in \{k_0, \dots, k_1 - 1 \}$. Since $k_1$ maximizes the cumulative mean,
    \begin{equation*}
        \frac{\sum_{j = k_0}^{N - 1} w_j}{N} \leq \frac{\sum_{j = k_0}^{k_1 - 1} w_j}{k_1 - k_0} \implies \frac{\sum_{j = k_1}^{N} w_j}{N - k_1 + 1} \leq \frac{\sum_{j = k_0}^{k_1 - 1} w_j}{k_1 - k_0}, 
    \end{equation*}
    implying that for some $k'' \in \{ k_1, \dots, N - 1 \}$,
    \begin{equation*}
        \tilde{u}_{k''} \leq \frac{\sum_{j = k_1}^{N} w_j}{N - k_1 + 1} \leq \frac{\sum_{j=k_0}^{k_1 - 1} w_j}{k_1 - k_0} < \tilde{u}_{k'}.
    \end{equation*}
    Since $G$ is strictly convex, there exists some positive $\varepsilon < \sum_{j = k_0}^{k_1 - 1} \tilde{u}_j - \tilde{u}_{k'}$ for which 
    \begin{equation*}
        G(\tilde{u}_{k'} - \varepsilon) + G(\tilde{u}_{k''} + \varepsilon) < G(\tilde{u}_{k'}) + G(\tilde{u}_{k''}). 
    \end{equation*}
    This implies the existence of another feasible solution with strictly lower cost, so $\tilde{\bm{u}}$ cannot be optimal. This contradiction establishes the claim.

    Using \eqref{eq:alg_reindex_constraints}, this reasoning can be extended to show that any optimal solution must meet the constraints corresponding to $k_2 - 1, \dots, k_{n} - 1$ with equality. However, if the constraint corresponding to $k_l - 1$ must be satisfied with equality, then $u_{k_0}, \dots, u_{k_l - 1}$ do not appear in any subsequent constraints; thus, the optimization problem can be decomposed into $n$ independent problems of the form
    \begin{equation*}
    \begin{aligned}
        \min_{\bm{u}} \quad & \sum_{k=k_l}^{k_{l+1} - 1} G(u_k) \\
        \textrm{s.t.} \quad & \sum_{j=k_l}^k u_j \geq \sum_{j=k_l}^k w_j, \quad \forall k \in \{k_l,\dots,k_{l+1} - 1\},
    \end{aligned}
    \end{equation*}
    for all $l \in \{ 0, \dots, n - 1 \}$. Since $G$ is convex, the optimal solution to each of these individual problems is the one that equates all of the decision variables while meeting the constraint strictly, i.e., $u_k = \frac{1}{k_{l+1} - k_l} \sum_{j=k_l}^{k_{l+1}-1} w_j$ , $\forall k \in \{k_l, \dots, k_{l+1} - 1\}$. This is precisely the solution returned by the algorithm, completing the proof.

    Last, we establish the monotonicity result. We proceed by contradiction. Let $\hat{\bm{u}}$ be an optimal solution to~\eqref{eq:SPO} for which $\hat{u}_{k'+1} > \hat{u}_{k'}$ for some $k' \in \{0,\dots,N-2\}$. Consider an alternative solution $\overline{\bm{u}}$ that is identical to $\hat{\bm{u}}$ except that $\overline{u}_{k'} = \overline{u}_{k'+1} = ( \hat{u}_{k'} + \hat{u}_{k'+1} ) / 2.$ The alternative solution $\overline{\bm{u}}$ is still feasible because $\overline{u}_{k'} > \hat{u}_{k'}$ and $\overline{u}_{k'} + \overline{u}_{k'+1} = \hat{u}_{k'} + \hat{u}_{k'+1}$. However, since $G$ is strictly convex, we have that
    \begin{equation*}
        G(\overline{u}_{k'}) + G(\overline{u}_{k'+1}) < G(\hat{u}_{k'}) + G(\hat{u}_{k'+1})
    \end{equation*}
    which is a contradiction.
\end{proof}

Algorithm~\ref{alg:load_shift} works by greedily identifying windows with high demand and equalizing order amounts within these windows to exploit convexity. In practice, this means that optimal trajectories typically remain flat for long periods of time and steadily decrease. While this algorithm is only provably optimal in the deterministic setting, it will still be valuable in the stochastic setting, as shown next.

\section{Stochastic Inventory Systems}\label{sect:stochastic}

In most inventory systems, including industrial refrigeration, one only has prior information regarding the distribution from which the demands are generated, making the deterministic solution impractical for direct implementation. In this section, we study and compare various strategies for inventory management in the face of uncertainty using the framework of inventory control.

\subsection{Literature Review} 

Optimal policies for inventory systems are well-known to be structurally simple in the case of linear ordering and convex holding/backlog cost functions~\cite{arrow1951optimal, scarf1960optimality}. However, for nonlinear ordering cost functions, (such as the convex function studied in this work), optimal policies typically have limited structure~\cite{perera2023survey, benjaafar2018optimal, veinott1964production}, and the special case of zero holding costs and infinite backlog costs considered in this work has received limited attention in the literature. Given that our focus is on the value of load shifting, and that characterizing optimal policies is both analytically and computationally challenging, the goal of this section is to assess the performance of simple tractable policies.

\subsection{Myopic Suboptimality}

We begin our study of stochastic inventory systems by analyzing the performance of the myopic policy. Recall that the myopic policy is the one that minimizes immediate costs without taking future demands into consideration. Our first result characterizes the suboptimality of the myopic policy.

\begin{theorem}\label{thm:lb_myopic}
    For any inventory system $P = (G, \bm{\mu}, \bm{\Sigma}, \Delta)$ that satisfies Assumptions \ref{assumption:convex_cost} and \ref{assumption:positive_demand} with $l \leq G''(x) \leq L$, the difference in expected costs between the myopic policy and the optimal policy is at least
    
    \vspace{0.25cm}
    
    \noindent \resizebox{\linewidth}{!}
    {$
        J_{\bm{\pi}^{\textup{my}}}(0) - J_{\bm{\pi}^*}(0) \geq \frac{l}{2}  \left( S_{\mu} + \sum_{k = 0}^{N - 2} \sigma_k^2 \right) - \frac{L}{2}  \left( S_{u} + \sum_{k = 0}^{N - 2} \sigma_k^2 \right) 
    $}

    \vspace{0.25cm}
    
    \noindent where
    
    \vspace{-0.3cm}
    
    \begin{gather*}
        M = \frac{\Delta + \sum_{k = 0}^{N - 1} \mu_k}{N}, \qquad \bm{u} = \text{\textsc{LoadShift}}(\bm{\mu},\Delta,0),
        \\
        S_{u} = \sum_{k = 0}^{N - 1} \left( u_k - M \right)^2, \\
        S_{\mu} = (\mu_0 + \Delta - M)^2 + \sum_{k = 1}^{N - 1} \left( \mu_k - M \right)^2.
    \end{gather*}
\end{theorem}

\begin{proof}
    See Appendix \ref{sec:lb_myopic_proof}.
\end{proof}

Theorem \ref{thm:lb_myopic} explicitly quantifies how much one stands to lose by not utilizing load shifting. This bound depends on
\begin{itemize}
    \item the convexity of the cost function $G$,
    \item the `spread' between the mean demands, and
    \item the magnitude of the variance.
\end{itemize}
Thus, in settings where any of these effects are severe, one should consider implementing a more advanced load shifting strategy that takes future demands into account. Note that in the special case where $G$ is quadratic, the lower bound depends only on the difference in spreads between the nominal means and the result of the load shifting algorithm. In the following section, we present a simple heuristic as an alternative to the myopic policy that performs load shifting effectively without significant computational overhead.

\subsection{An Effective Heuristic}

The previous section establishes that load shifting can offer significant benefits, but solving for the optimal load shifting policy $\bm{\pi}^*$ can be difficult, since optimal policies often have little discernible structure (i.e., they cannot be uniquely characterized by specific `order-up-to' levels \cite{veinott1964production}). Thus, in this section, we present a simple heuristic as an alternative to $\bm{\pi}^*$ and characterize its expected costs.

\begin{algorithm}
\caption{Load shifting heuristic.}\label{alg:LSH}
\begin{algorithmic}
\Procedure{LSH($\bm{\mu}$, $\Delta$)}{}
\State $x'_0 \gets \mu_0 + \Delta$
\State $\bm{u}_{1:N-1} \gets \text{\textsc{LoadShift}}(\bm{\mu}_{1:N-1},0,0)$
\For{$k \in \{1,\dots,N-1\}$}
\State $x'_k \gets x'_{k-1} + u_k - \mu_{k-1}$
\EndFor
\State \Return $\bm{x}'$
\EndProcedure
\end{algorithmic}
\end{algorithm}

The load shifting heuristic we propose is shown in Algorithm~\ref{alg:LSH} and relies on the load shifting algorithm from the deterministic setting. 
Here, we use the notation $\bm{\mu}_{1 : N-1}$ to refer to the subsequence $\{ \mu_1, \dots, \mu_{N - 1} \}$, and similarly for $\bm{u}_{1 : N-1}$. The heuristic policy is defined by
\begin{equation}
    \bm{\pi}_k^{\textup{h}} \triangleq \max \{0, x_k' - x_k \}
\end{equation}
for all $k$. The policy aims to track the behavior of the optimal load shifting algorithm in the deterministic setting. The following theorem bounds the difference between the expected costs of the heuristic and the optimal policy.

\begin{theorem}\label{thm:ub_heuristic}
    For any inventory system $P = (G, \bm{\mu}, \bm{\Sigma}, \Delta)$ that satisfies Assumptions \ref{assumption:convex_cost} and \ref{assumption:positive_demand} with $G''(x) \leq L$, the difference in expected costs between the heuristic and the optimal policy is at most
    \begin{align}
        J_{\bm{\pi}^{\textup{h}}}(0) - J_{\bm{\pi}^*}(0) &\leq \frac{L}{2}  \left( S_{u} + \sum_{k = 0}^{N - 2} \sigma_k^2 \right)
    \end{align}
    where
    \begin{gather*}
        M = \frac{\Delta + \sum_{k = 0}^{N - 1} \mu_k}{N}, \qquad \bm{u} = \text{\textsc{LoadShift}}(\bm{\mu},\Delta,0),
        \\
        S_{u} = \sum_{k = 0}^{N - 1} \left( u_k - M \right)^2.
    \end{gather*}
\end{theorem}

\begin{proof}
    See Appendix \ref{sec:ub_heuristic_proof}.
\end{proof}

Theorem \ref{thm:ub_heuristic} establishes that the suboptimality of the heuristic, similar to that of the myopic policy, depends on the convexity of the function, the variance of the demand, and the spread of the mean demands. However, notice that the suboptimality in this setting depends on the spread between the output of the load shifting algorithm, and not the nominal means themselves, which can be significantly larger.

\section{Simulations}


\begin{figure*}
\centering
\includegraphics[width=\linewidth]{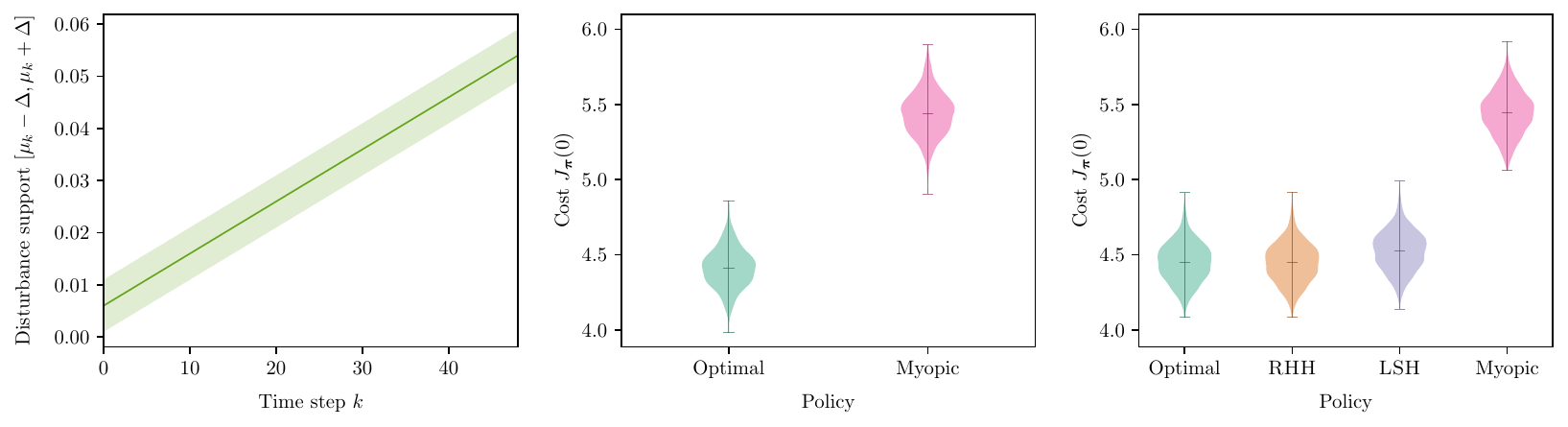}
\caption{Cost comparisons of various policies in deterministic and stochastic settings over 1,000 simulations. Left: the trajectory of the demand support; we consider a discrete uniform probability distribution at every time step. Center: the performance of the optimal and myopic policies in the deterministic setting; the \textit{known} demand trajectories $\bm{w}$ are generated by sampling the distribution (left) a priori. Right: the performance of the optimal policy, both heuristics, and the myopic policy in the stochastic setting. The known distributions (left) are used to compute the optimal policy via dynamic programming. All four policies are then evaluated via Monte Carlo simulation. The simulations use a state space $x \in [0,1]$, discretization $dx = 0.001$, horizon $N = 50$, cost function $G(u_k) = 100 u_k^2$, and $w_k$ discretized like the state.}
\label{fig:sims}
\end{figure*}

Figure~\ref{fig:sims} shows simulation results for the deterministic setting of Section~\ref{sect:deterministic} and the stochastic setting of Section~\ref{sect:stochastic}. In the deterministic setting, the optimal suction pressure trajectory from Algorithm~\ref{alg:load_shift} leads to an average cost reduction of $18.9$\% compared to the myopic policy. In the stochastic setting, this cost reduction is $18.3$\%. For this problem, our theoretical lower bound for the average performance gap $J_{\bm{\pi}^{\textup{my}}}(0) - J_{\bm{\pi}^*}(0)$ and the observed average gap are $0.999$ and $0.976$, respectively. For the load shifting heuristic we observe a $16.9$\% reduction in average cost compared to the myopic policy. Our theoretical upper bound for the average performance gap $J_{\bm{\pi}^{\textup{h}}}(0) - J_{\bm{\pi}^*}(0)$ and the observed average gap are $0.087$ and $0.077$, respectively. Note that the theoretical bounds for the myopic and heuristic policies are relatively close to the observed gaps.

As shown in Figure~\ref{fig:sims}, we also propose the receding horizon heuristic shown in Algorithm~\ref{alg:RHH} that achieves excellent empirical performance. At each time step, the algorithm is re-executed and only the first control action is implemented. 

\begin{algorithm}
\caption{Receding horizon heuristic.}\label{alg:RHH}
\begin{algorithmic}
\Procedure{RHH($\bm{\mu}_{k:N-1}$, $\Delta$, $x_k$)}{}
\State $\bm{u}_{k:N-1} \gets \text{\textsc{LoadShift}}(\bm{\mu}_{k:N-1},\Delta,x_k)$
\State \Return $u_k$
\EndProcedure
\end{algorithmic}
\end{algorithm}


\begin{figure}
    \centering
    \includegraphics[width=0.97\linewidth]{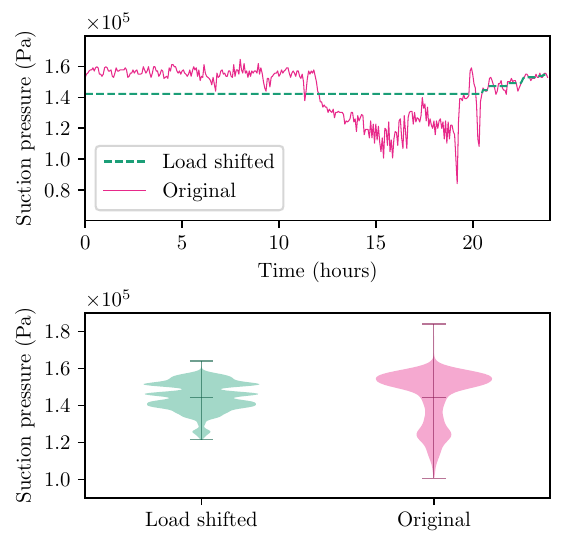}
    \caption{A case study in suction pressure control using real data from February 2025 at an industrial refrigeration facility in Salem, CT. The month of data is broken up into 24-hour windows. Top: original and optimized (from Algorithm~\ref{alg:load_shift}) suction pressure trajectories over one sample window. Bottom: violin plot of suction pressure data and corresponding optimized values over all windows. Note the reduction in variability after load shifting.}
\label{fig:case_study}
\end{figure}

Returning to our motivation of suction pressure control in industrial refrigeration, we also examine a month of real suction pressure data from a facility in Salem, CT as a case study. To analyze the data, we split it into 24-hour windows and compute the heat absorption $u_k = H\left(P^s_k\right)$ at each time step $k$ from the original suction pressure trajectory $\bm{P}^s$. We compare that heat absorption trajectory $\bm{u}$ with the hindsight-optimal deterministic solution $\bm{u}^*$ to~\eqref{eq:SPO} from Algorithm~\ref{alg:load_shift}. For a fair comparison, we choose the heat load trajectory $\bm{w} = \bm{u}$ to ensure that the optimized trajectory $\bm{u}^*$ must perform the same cumulative heat absorption. 
Fig.~\ref{fig:case_study} shows an example window of real suction pressure data and the corresponding optimized trajectory; observe that the optimized profile maintains a relatively consistent level throughout the day. Fig.~\ref{fig:case_study} also shows that over the course of a month, the optimized suction pressure trajectories have meaningfully less variance, indicating potential cost savings. 

\section{Conclusion}

This paper draws a connection between industrial refrigeration systems and inventory control problems characterized by convex ordering costs and non-negativity constraints. By recognizing the convex relationship between compressor work and heat removal, we develop both optimal algorithms for deterministic settings and practical heuristics with provable performance guarantees for stochastic settings. Our theoretical analysis quantifies the value of load shifting strategies, showing that the benefits depend on the degree of cost convexity and demand variability. The numerical simulations and the case study using real industrial refrigeration data validate these insights, indicating a potential for cost savings through optimized suction pressure control. Future directions include incorporating additional factors in our model such as time-varying discharge pressure and electricity prices.

\appendix

\subsection{Expected Cost of Myopic Policy}

In period $0$, the myopic policy must meet the constraint
\begin{align*}
    u_0 \geq \max\{ 0, \mu_0 + \Delta - x_0 \} = \mu_0 + \Delta
\end{align*}
so it will simply order $u_0 = \mu_0 + \Delta \implies x_1 = \mu_0 + \Delta - w_0$. In the second period, it must meet the constraint
\begin{align*}
    u_1 &\geq \max\{ 0, \mu_1 + \Delta - x_1 \} \\
    &= \max\{ 0, \mu_1 + \Delta - (\mu_0 + \Delta - w_0) \} = \mu_1 - \mu_0 + w_0,
\end{align*}
so it will place an order for exactly this amount; the last line follows since $\mu_1 > \Delta \geq \mu_0 - w_0$. Thus, we have
\begin{align*}
    x_2 &= x_1 + u_1 - w_1 \\
    &= \mu_0 + \Delta - w_0 + \mu_1 - \mu_0 + w_0 - w_1 \\
    &= \mu_1 + \Delta - w_1
    \implies u_2 = \mu_2 - \mu_1 + w_1.
\end{align*}
It is readily verified that this reasoning extends to all $k$. Thus, the total ordering cost for any realization of demands is
\begin{equation*}
    G(\mu_0 + \Delta) + \sum_{k = 1}^{N - 1} G \left( \mu_k - w_{k - 1} - \mu_{k-1} \right).
\end{equation*}
Taking expectations yields that $J_{\bm{\pi}^{\textup{my}}} (0 \, | \, P)$ equals
\begin{equation}\label{eq:expected_costs_myopic}
    G(\mu_0 + \Delta) + \sum_{k = 1}^{N - 1} \mathbb{E} [G \left( \mu_k + w_{k - 1} - \mu_{k-1} \right)].
\end{equation}

\subsection{Expected Cost of Heuristic}

First, we show that if the heuristic orders $x_k' - x_k$ in each period, then the constraint $x_k \geq 0,\ \forall k$ is satisfied. The first-period constraint
\begin{equation*}
    x_1 = x_0 + u_0 - w_0 = \mu_0 + \Delta - w_0 \geq 0
\end{equation*}
is clearly satisfied, since $w_0 \leq \mu_0 + \Delta$. In all subsequent periods, observe that we can rewrite $x_k'$ as $\sum_{j = 0}^k u_j - \sum_{j = 0}^{k - 1} \mu_j$. Thus, assuming $x_k' - x_k > 0$, we have
\begin{equation*}
    x_{k + 1} = x_k + u_k - w_k = x_k' - w_k
\end{equation*}
which can be rewritten as
\begin{align*}
    \sum_{j = 0}^k u_j - \sum_{j = 0}^{k - 1} \mu_j - w_k
    &\geq \Delta + \sum_{j = 0}^k \mu_j - \sum_{j = 0}^{k - 1} \mu_j - w_k \\
    &= \mu_k + \Delta - w_k \geq 0,
\end{align*}
so the constraint is satisfied for all $k$. What remains to be shown is that the heuristic orders $x_k' - x_k$ in each period, meaning $x_k' - x_k \geq 0$ for all $k$. For $k = 0$, the heuristic orders $x_0' = \mu_0 + \Delta$, so $x_0' > x_0$. For $k = 1$, we have
\begin{equation*}
    x_1' - x_1 = (u_0 + u_1 - \mu_0) - (u_0 - w_0) = u_1 - \mu_0 + w_0.
\end{equation*}
By construction, $u_1$ is the average of $\mu_j, \dots, \mu_{j + i}$ for some $j$, $i$, and since $\mu_k - \Delta > 0$ for all $k$, we must also have that $u_1 - \mu_0 + w_0 > u_1 - \Delta > 0$. Thus, we have
\begin{align*}
    \max \{0, x_1' - x_1 \}& = x_1' - x_1 \\
    \implies u_1 &= x_1' - x_1 = u_1 - \mu_0 + w_0 \\
    \implies x_2 &= x_1 + u_1 - w_1 = x_1' - w_1.
\end{align*}
It is readily verified that for all $k \in \{2, \dots, N - 1\}$,
\begin{align*}
    x_k' - x_k &= u_k - \mu_{k - 1} + w_{k - 1} \geq 0 \\
    x_k &= x_{k - 1}' - w_{k - 1},
\end{align*}
so $x_k' - x_k$ is ordered for all $k \geq 1$ and all constraints are met. Thus, the total ordering costs are given by
\begin{equation*}
    G( \mu_0 + \Delta) + \sum_{k = 1}^{N - 1} G \left( u_k - \mu_{k - 1} + w_{k - 1} \right).
\end{equation*}
Taking expectations yields that $J_{\bm{\pi}^{\textup{h}}} (0 \, | \, P)$ equals
\begin{equation}\label{eq:expected_costs_heuristic}
    G(\mu_0 + \Delta) + \sum_{k = 1}^{N - 1} \mathbb{E} [G \left( u_k +  w_{k - 1} - \mu_{k-1} \right)].
\end{equation}

\subsection{Proof of Theorem \ref{thm:ub_heuristic}}\label{sec:ub_heuristic_proof}

The second order Taylor expansion of $G$ about $u_k$ is
\begin{align*}
    G(y) &= G(u_k) + G'(u_k)(y - u_k) + \frac{1}{2} G''(\zeta(y)) (y - u_k)^2 \\
    &\leq G(u_k) + G'(u_k)(y - u_k) + \frac{L}{2} (y - u_k)^2
\end{align*}
where $\zeta(y)$ lies between $y$ and $u_k$. Taking the expectation and evaluating at $y = u_k - \mu_{k - 1} +  w_{k - 1}$ yields
\begin{align*}
    \mathbb{E} \bigg[ &G(u_k) + G'(u_k) (w_{k - 1} - \mu_{k-1}) + \frac{L}{2} (w_{k - 1} - \mu_{k-1})^2 \bigg] \\
    = \, &G(u_k) + \frac{L}{2} \text{Var}(w_{k - 1}).
\end{align*}
Substituting this expression into \eqref{eq:expected_costs_heuristic}, the expected costs under the heuristic policy can be upper bounded by 
\begin{equation*}
    J_{\bm{\pi}^{\textup{h}}}(0) \leq \sum_{k = 0}^{N - 1} G (u_k) + \frac{L}{2} \sum_{k = 0}^{N - 2} \sigma_k^2,
\end{equation*}
where $\sigma_k^2 = \text{Var}(w_{k})$.
We take another second-order Taylor expansion about $M$ and evaluate it at $u_k$, which yields
\begin{align*}
    \sum_{k = 0}^{N - 1} G (u_k) = &\sum_{k = 0}^{N-1} \Big( G\left( M \right) + G'(M) (u_k - M) \\ + \, &\frac{1}{2} G''(\eta(u_k))(u_k - M)^2 \Big) \leq N G(M) + \frac{L}{2} S_u.
\end{align*}
Here, observe that the first order terms vanish because $\sum_{k = 0}^{N - 1} u_k = \sum_{k = 0}^{N - 1} \mu_k + \Delta = N M$. For the optimal policy, observe that in period $k = 0$, the optimal policy must order at least $\mu_0 + \Delta$ units, and across all subsequent periods, it must order at least $\sum_{k = 1}^{N - 1} w_k$ units. Thus, for any realization, the cost of the optimal policy is lower bounded by 
\begin{align*}
    G(\mu_0 + \Delta) + \sum_{k = 1}^{N - 1} G(w_k) \geq N G \left( \frac{\mu_0 + \Delta + \sum_{k = 1}^{N - 1} w_k}{N} \right).
\end{align*}
Taking expectations yields
\begin{align*}
    J_{\bm{\pi}^*}(0) &\geq N \mathbb{E} \left[ G \left( \frac{\mu_0 + \Delta + \sum_{k = 1}^{N - 1} w_k}{N} \right) \right] \\
    &\geq N G \left( \mathbb{E} \left[ \frac{\mu_0 + \Delta + \sum_{k = 1}^{N - 1} w_k}{N} \right] \right) = N G(M).
\end{align*}
Subtracting $N G(M)$ from the upper bound on the heuristic policy yields the upper bound
$\frac{L}{2} \left( S_{u} + \sum_{k = 0}^{N - 2} \sigma_k^2 \right)$
on the difference $J_{\bm{\pi}^{\textup{h}}}(0) - J_{\bm{\pi}^*}(0)$.

\subsection{Proof of Theorem \ref{thm:lb_myopic}}\label{sec:lb_myopic_proof}

Since \eqref{eq:expected_costs_myopic} and \eqref{eq:expected_costs_heuristic} are similar, the proof of Theorem \ref{thm:lb_myopic} is similar to that of Theorem \ref{thm:ub_heuristic}. Notice that the bound
\begin{equation*}
    J_{\bm{\pi}^{\textup{my}}}(0) \geq \frac{l}{2} \left( S_{\mu} + \sum_{k = 0}^{N - 2} \sigma_k^2 \right)
\end{equation*}
can be derived in a near identical manner to the upper bound on $J_{\bm{\pi}^{\textup{h}}}(0)$ for Theorem \ref{thm:lb_myopic} with $l$ in place of $L$ and $S_\mu$ in place of $S_u$. Then, since $J_{\bm{\pi}^{\textup{h}}}(0) \geq J_{\bm{\pi}^{\textup{*}}}(0)$, we have
\begin{gather*}
    J_{\bm{\pi}^{\textup{my}}}(0) - J_{\bm{\pi}^*}(0) \geq \, J_{\bm{\pi}^{\textup{my}}}(0) - J_{\bm{\pi}^{\textup{h}}}(0),
    \intertext{which is lower bounded by}
    \frac{l}{2} \left( S_{\mu} + \sum_{k = 0}^{N - 2} \sigma_k^2 \right)
    - \, \frac{L}{2} \left( S_{u} + \sum_{k = 0}^{N - 2} \sigma_k^2 \right),
\end{gather*}
yielding the desired result.

\bibliographystyle{plain}
\bibliography{ref}

\end{document}